\documentclass[]{interact}

\usepackage{epstopdf}
\usepackage[caption=false]{subfig}
\usepackage{wrapfig}
\usepackage{graphicx}
\usepackage[longnamesfirst,sort]{natbib}
\bibpunct[, ]{(}{)}{;}{a}{,}{,}
\renewcommand\bibfont{\fontsize{10}{12}\selectfont}

\theoremstyle{plain}
\newtheorem{theorem}{Theorem}[section]
\newtheorem{lemma}[theorem]{Lemma}

\theoremstyle{definition}
\newtheorem{definition}[theorem]{Definition}

\theoremstyle{remark}
\newtheorem{remark}{Remark}

\newtheorem{assumption}{Assumption}
\begin{document}


\title{Performance Bounds for  Nash Equilibria in Submodular Utility Systems with User Groups}

\author{
\name{Yajing Liu\textsuperscript{a}\thanks{CONTACT Yajing Liu. Email: yajing.liu@ymail.com},  Edwin K. P. Chong\textsuperscript{b} and Ali Pezeshki\textsuperscript{b}}
\affil{\textsuperscript{a} Department of Electrical and Computer Engineering, Colorado State University, Fort~Collins, USA; \textsuperscript{b} Department of Electrical and Computer Engineering \& \\Department of Mathematics, Colorado State University, Fort Collins, USA}
}

\maketitle

\begin{abstract}
\cite{Vetta2002} shows that for a valid non-cooperative utility system, if the social
utility function is submodular, then any Nash equilibrium
achieves at least $1/2$ of the optimal social utility, subject to a
function-dependent additive term. Moreover, if the social
utility function is nondecreasing and submodular, then any
Nash equilibrium achieves at least $1/(1+c)$ of the optimal social
utility, where $c$ is the curvature of the social utility function. In this paper, we consider variations of the utility system considered by Vetta, in which users are grouped together. Our aim is to establish how grouping and cooperation among users affect performance bounds.  We consider two types of grouping. The first type is from  \cite{Zhang2014}, where each user
belongs to a group of users having social ties with it. For this type of
utility system, each user's strategy maximizes its social group utility
function, giving rise to the notion of \emph{social-aware Nash
equilibrium}.  We prove that this social utility system yields to the bounding results of Vetta  for non-cooperative system, thus establishing provable performance guarantees for the  social-aware Nash equilibria. 
For the second type of grouping we consider, the set of users is
partitioned into $l$ disjoint groups, where the users within a group
cooperate to maximize their group utility function, giving rise to the
notion of \emph{group Nash equilibrium}. In this case, each group can be viewed as a new user with vector-valued actions, and a 1/2 bound for the performance of group Nash equilibria follows from the result of Vetta. But as we show tighter bounds involving curvature can be established.
 By defining the group
curvature $c_{k_i}$ associated with group $i$ with $k_i$ users, we
show that if the social utility function is nondecreasing and
submodular, then any group Nash equilibrium achieves at least
$1/(1+\max_{1\leq i\leq l}c_{k_i})$ of the optimal social utility, which is tighter than that for the case without grouping. As a special case, if each user has the same action space, then we have that any group Nash equilibrium achieves at least $1/(1+c_{k^*})$ of the optimal social utility, where $k^*$ is the least number of users among the $l$ groups.
Finally, we present an example of a utility system for database assisted
spectrum access to illustrate our results.
\end{abstract}

\section{Introduction}

A variety of interesting practical problems can be posed as utility maximization problems: these include facility location (e.g., \cite{Atamturk2011}), traffic routing and congestion management (e.g., \cite{Marden2007} and \cite{Rexford2007}), sensor selection (e.g., \cite{Portal2007} and \cite{Chong2014}), and network resource allocation (e.g., \cite{Anantharam2002}, \cite{Chiang2007}, \cite{Zhang2014}, and \cite{Lu2012}). In a utility maximization problem, a set of users make decisions according to their own set of feasible strategies, resulting in an overall social utility value, such as profit, coverage, achieved data rate, and quality of service. The goal is to maximize the social utility function. Often, the users do not cooperate in selecting their strategies.

In general, it is impractical to find the globally optimal sequence (finite, ordered collection) of strategies maximizing the social utility function. Typically, it is more useful to consider scenarios where individual users or groups of users separately maximize their own \emph{private} objective functions, and then ask how this compares with the globally optimal case. The usual framework for studying such scenarios is game theory together with its celebrated notion of Nash equilibria.
A \emph{Nash equilibrium} is a sequence of strategies (deterministic or randomized) for which no user can improve its own private utility by changing its strategy unilaterally. \cite{Nash1951} proves that any finite and non-cooperative game has at least one Nash equilibrium. 

The question of how the Nash solution compares with the globally optimal solution is one of the most challenging problems in game theory and has received significant attention in the literature (see the survey by \cite{Papadimitriou2001}). For example, results have been reported by \cite{Papadimitriou2009} and \cite{Tardos2002} in the context of traffic routing and congestion management, which aims to minimize the total latency. For a general utility maximization problem, \cite{Vetta2002} develops lower bounds on the worst-case social utility value in non-cooperative games. Specifically, he proves that for a \emph{submodular} social utility function in a \emph{valid} utility system, any Nash equilibrium achieves at least $1/2$ of the optimal social utility value, subject to a function-dependent additive term. He also shows that for a nondecreasing and submodular social utility function in a valid utility system, any Nash equilibrium can achieve at least $1/(1+c)$ of the optimal social utility function value, where $0\leq c\leq 1$ is the \emph{curvature} of the social utility function. 

With the advent of social networks, there is increasing interest in understanding the role of cooperation and social ties in games (see, e.g., the recent paper by \cite{Allen2017}). In our paper, we are interested in exploring bounds for Nash equilibria when there is some notion of ``grouping'' among users. Along these lines, we consider two notions of grouping that yield to provable performance bounds. The first type of grouping we consider is the recent framework of  \cite{Zhang2014}, where associated with each user is a private objective function and a fixed group of users having some social ties with it. Each user's strategy maximizes an objective function called the \emph{social group utility}, which is the sum of its private objective function and a linear combination of the private objective functions of users in its group. Within this setting, \cite{Zhang2014} define what they call a \emph{social-aware Nash equilibrium}, where no user can improve its social group utility by unilaterally changing its strategy.  We will show that this framework yields to the bounding results of \cite{Vetta2002} for noncooperative games, thus establishing provable performance guarantees for the framework of \cite{Zhang2014}.

In the second type of grouping we consider, the set of users is partitioned into disjoint groups. Associated with each group is a group utility function. Users within a group \emph{cooperate} in the sense that their strategy is to (jointly) maximize the group utility function, giving rise to a natural definition of \emph{group Nash equilibrium}. 
Although we can view each group as a new user with vector-valued actions so that a similar $1/2$ bound to the result of \cite{Vetta2002} holds, we would like to investigate the performance bound for the group Nash equilibrium in terms of curvature and compare it with the case where there is no grouping. We define a measure of \emph{group curvature} and derive an associated lower bound involving this curvature. We prove that this bound is tighter than that for the case without grouping among users, accounting for the cooperation within the groups. We also prove that, under the condition that each user has the same action space, the higher the degree of cooperation, the tighter the lower bound.

The remainder of the paper is organized as follows. In Section~2, we  introduce our notation and some definitions that will be used throughout the paper.
In Section~3, we review the bounding results of \cite{Vetta2002}. In Section~4, we first describe the framework of \cite{Zhang2014} and show that a social-aware utility system yields to the bounding results of Vetta  for non-cooperative system, thus establishing provable performance guarantees for the  social-aware Nash equilibrium. Next, we describe our second type of grouping involving $l$ disjoint groups with in-group cooperation. In this case, each group can be viewed as a new user with vector-valued actions, and a 1/2 bound for the performance of group Nash equilibrium follows from the result of \cite{Vetta2002}. We then define the group curvature $c_{k_i}$ associated with group $i$ with $k_i$ users, and we show that if the social utility function is nondecreasing and submodular, then any group Nash equilibrium achieves at least $1/(1+\max_{1\leq i\leq l}c_{k_i})$ of the optimal social utility, which is tighter than that for the case without grouping.  Especially, if each user has the same action space, then we have that any group Nash equilibrium achieves at least $1/(1+c_{k^*})$ of the optimal social utility, where $k^*$ is the least number of users among all the groups. In Section~5, we present an example of a utility system for database assisted spectrum access, adopted from \cite{Zhang2014}. We show that the utility system for this example is valid and the social utility function is submodular, illustrating an application of our results.

\section{Preliminaries}

In this section, we first introduce notation and a number of definitions used throughout the paper.
\subsection{Actions}

Suppose we have a set  $\mathcal{N}=\{1,2,\ldots,N\}$ of $N$ users and ground sets $V_1,V_2,\ldots,V_N$, where each element in $V_i$ denotes an \emph{act} that user $i$ can take. We call a set of acts an \emph{action}, and if an action $x_i\subseteq V_i$ is available to user $i$ we call it a \emph{feasible action}. We denote by $\mathcal{X}_i$ the set of all feasible actions for user $i$, i.e., $\mathcal{X}_i=\{x_i\subseteq V_i: x_i$ is a feasible action$\}$, with $n_i=|\mathcal{X}_i|$ the cardinality of $\mathcal{X}_i$.

Let $\mathcal{X}=\prod_{i=1}^N\mathcal{X}_i$ and $X=(x_{i_1},\ldots,x_{i_k})$, where $x_j\in\mathcal{X}_j$, with $i_1\leq j\leq i_k$. We call $X$ an \emph{action sequence} of length $k$ in $\mathcal{X}$. This sequence includes the actions taken by users $i_1,\ldots, i_k$ in order. Given an action sequence $X$, suppose $Y$ is formed by removing some of the elements of $X$ without changing the order of the remaining elements. Then, we call the derived action sequence $Y$ a \emph{subsequence} of $X$ and denote this relation by $Y\subseteq X$. This follows the definition of a subsequence in \cite{Smith2015}.   

Consider an action sequence $X=(x_1,\ldots,x_N)\in\mathcal{X}$. Then, $X_{-i}=(x_1,\ldots,x_{i-1},x_{i+1},\ldots,x_N)$ is the subsequence of $X$ that includes actions taken by all users except user $i$. We use $(X_{-i},x_{i}')$ to denote the action sequence $(x_1,\ldots,x_{i-1},x_i',x_{i+1},\ldots,x_N)$ that results from $X$ when user $i$ changes its action from $x_i$ to $x_i'$. 

Given action sequences $Y=(y_{i_1},\ldots, y_{i_k})$  and $Z=(z_{j_1},\ldots, z_{j_l})$, we define $Y\oplus Z=(y_{i_1},\ldots, y_{i_k},z_{j_1},\ldots, z_{j_l})$ as the concatenation of $Y$ and $Z$ when $i_p\neq j_q$ for  $1\leq p\leq k$ and $1\leq q \leq l$ 
(following the notation in \cite{Zhang2016}).

\subsection{Strategies}

Let $s_i=(s_i^1,\ldots, s_i^{n_i})$, where $s_i^j\ge 0$ is the probability with which user $i$ takes action $j$ and  $\sum_{j=1}^{n_i}s_i^j=1$. Following the terminology of \cite{Vetta2002}, we call $s_i$ a \emph{strategy} taken by user $i$. When $s_i^j=1$ and $s_i^l=0$ for $1\leq j\leq n_i$ and $l\neq j$, we say that user $i$ takes a \emph{pure strategy}. Otherwise, we say that user $i$ takes a \emph{mixed strategy}.

Let $\mathcal{S}_i=\{s_i\in\mathbb{R}_i^{n_i}:\sum_{j=1}^{n_i}s_i^j=1,s_i^j\geq 0\}$ be the strategy space for user $i$
and $\mathcal{S}=\prod_{i=1}^N\mathcal{S}_i$. Similar to the definition of an action sequence, we call $S=(s_{i_1},\ldots, s_{i_k})$, with $s_j\in\mathcal{S}_j$ and $i_1\leq j\leq i_k$, a \emph{strategy sequence} of length $k$ in $\mathcal{S}$. Then a subsequence $T$ of $S$ is a sequence derived from $S$ by deleting some elements without changing the order of the remaining elements. We define $S_i=(s_1,\ldots, s_i)$, for $1\leq i\leq N$, as a sequence of strategies taken by users $1,\ldots, i$.

Given a strategy sequence $S=(s_1,\ldots,s_N)\in\mathcal{S}$, the sequence $S_{-i}=(s_1,\ldots,s_{i-1},s_{i+1},\ldots,s_N)$ is the subsequence of $S$ that contains strategies taken by all users except user $i$, and $(S_{-i},s_{i}')=(s_1,\ldots,s_{i-1},s_i',s_{i+1},\ldots,s_N)$ is the strategy sequence that results from $S$ when user $i$ changes its strategy from $s_i$ to $s_i'$. 

Given strategy sequences $T=(t_{i_1},\ldots, t_{i_k})$  and $W=(w_{j_1},\ldots, w_{j_l})$, we write $T\oplus W=(t_{i_1},\ldots, t_{i_k},w_{j_1},\ldots, w_{j_l})$ for the concatenation of $T$ and $W$ when $i_p\neq j_q$ for $1\leq p\leq k$ and $1\leq q \leq l$.

\subsection{Utility Functions}

We define the \emph{social utility} function as a mapping $\gamma$ from sequences in $\mathcal{X}$ to real numbers, and the \emph{private utility} function for user $i $ $(1\leq i\leq N)$ as a mapping $\alpha_i$ from sequences in $\mathcal{X}$ to real numbers.
Correspondingly, we define $\bar{\gamma}$ and $\bar{\alpha}_i$ as mappings, from sequences in $\mathcal{S}$ to real numbers, that correspond to the expectations of $\gamma$ and $\alpha_i$, respectively. We call $\bar{\gamma}$ the \emph{expected social utility function} and $\bar{\alpha_i}$ the \emph{expected private utility function} for user $i$. We also define 
$\gamma_Z(Y)=\gamma(Y\oplus Z)-\gamma(Y)$ for any $Y, Z$ in $\mathcal{X}$ such that $Y\oplus Z$ is well defined, and $\bar{\gamma}_W(T)=\bar{\gamma}(T\oplus W)-\bar{\gamma}(T)$ for any $T, W$ in $\mathcal{S}$ such that $T\oplus W$ is defined. 

We denote by $\Omega$ the optimal sequence of strategies in maximizing an expected utility function $\bar{\gamma}$, and assume that $\Omega=(\sigma_1,\ldots,\sigma_N)$ is composed of pure strategies $\sigma_i\in\mathcal{S}_i$, $i=1,\ldots, N$. For convenience, we also use $\sigma_i$ to denote the optimal action that user $i$ takes. Then, we have that the optimal value of $\bar{\gamma}$, denoted by OPT, is OPT~$=\bar{\gamma}(\Omega)=\gamma(\Omega)$. 

\subsection{Curvature, Monotoneity, and Submodularity}

Given a strategy sequence $S_i=(s_1,\ldots,s_i)$ for $1\leq i\leq N$, we use the notation $\Omega\cup S_i$ to represent the sequence in which user $j$ $(1\leq j\leq i)$ implements the actions $\sigma_j\cup x_j^1,\ldots,\sigma_j\cup x_j^{n_j}$ with probabilities $s_j^1,\ldots,s_j^{n_j}$, and user $j$ $(j>i)$ plays the action $\sigma_j$, so $\bar{\gamma}(\Omega\cup S_i)$ is well defined.
Then the curvature $c$ of the expected social utility function $\bar{\gamma}$ is defined as 
\[c=\max_{i:\bar{\gamma}_{s_i}(\emptyset)\neq 0}\left\{1-\frac{\bar{\gamma}_{s_i}(\Omega\cup S_{-i})}{\bar{\gamma}_{s_i}(\emptyset)}\right\}.\] 

The social utility function $\gamma$ is called \emph{nondecreasing} if for all subsequences $Y$ of a sequence $X$ in $\mathcal{X}$, i.e., $Y\subseteq X$ in $\mathcal{X}$, $f(Y)\leq f(X)$. It is called \emph{submodular} if for all $Y\subseteq X$ and $Z$ in $\mathcal{X}$ such that $X\oplus Z$ is defined, we have $\gamma_Z(Y)\geq \gamma_Z(X)$. Our terminology here is consistent with that of \cite{Smith2015}. Because $\bar{\gamma}$ is the expected value of $\gamma$, we have that if $\gamma$ is nondecreasing and submodular, then $\bar{\gamma}$ is also nondecreasing and submodular, respectively. So in the following sections, when we say that $\gamma$ is nondecreasing and submodular, it implies that $\bar{\gamma}$ is nondecreasing and submodular, respectively.

\section{Performance Bounds for Nash Equilibria}

In this section, we first review the definitions of a Nash equilibrium and a valid utility system from \cite{Vetta2002}. We then review the bounds derived in \cite{Vetta2002} for the performance of any Nash equilibrium.

\begin{definition}
\label{dfn:NE}
A strategy sequence $S\in\mathcal{S}$ is a \emph{Nash equilibrium} if no user has an incentive to unilaterally change its strategy, i.e., for any user $i$, 
\begin{equation}
\label{ineq:NE}
\bar{\alpha}_i(S)\geq \bar{\alpha}_i((S_{-i}, s_i')),\quad \forall s_i'\in\mathcal{S}_i.
\end{equation}
\end{definition}
\medskip
\begin{assumption}
\label{assumption1}
\cite{Vetta2002} The private utility of user $i$ ($1\leq i\leq N$) is at least as large as the loss in the social utility resulting from user $i$ dropping out of the game. That is, the system ($\bar{\gamma}, \{\bar{\alpha}_i\}_{i=1}^N$) has the property  that  for any strategy sequence $S=(s_1,\ldots, s_N)\in\mathcal{S}$,
\begin{equation}
\label{assum1}
\bar{\alpha}_i(S)\geq \bar{\gamma}_{s_i}(S_{-i}),\quad \forall 1\leq i\leq N.
\end{equation}
\end{assumption}

\medskip
\begin{assumption}
\label{assumption2}
\cite{Vetta2002} The sum of the private utilities of the system is not larger than the social utility, i.e., for any strategy sequence $S=(s_1,\ldots, s_N)\in\mathcal{S}$,
\begin{equation}
\label{assum2}
\sum\limits_{i=1}^N\bar{\alpha}_i(S)\leq\bar{\gamma}(S) .
\end{equation}
\end{assumption}

A utility system $(\gamma, \{\alpha_i\}_{i=1}^N)$ satisfying Assumptions~\ref{assumption1} and~\ref{assumption2} is called a \emph{valid} system. Given $X\in\mathcal{X}$, if for any $ 1\leq i\leq N$, the inequalities $\alpha_i(X)\geq\gamma_{x_i}(X_{-i})$  and $\sum_{i=1}^N\alpha_i(X)\leq \gamma(X)$ hold, then the inequalities (\ref{assum1}) and (\ref{assum2}) hold.
\begin{theorem}
\label{thm:NE}
\cite{Vetta2002} For a valid utility system $(\gamma, \{\alpha_i\}_{i=1}^N)$, if the social utility function $\gamma$ is submodular, then for any Nash equilibrium $S\in\mathcal{S}$ we have 
\begin{equation}
\label{ineq:NE}
\bar{\gamma}(S)\geq\frac{1}{2}\left(\bar{\gamma}(\Omega)+\sum_{i=1}^N\bar{\gamma}_{s_i}({S_{-i}\cup \Omega})\right).
\end{equation}
\end{theorem}

If $\gamma$ is non-decreasing, then $\bar{\gamma}_{s_i}({S_{-i}\cup \Omega})\geq 0$ and the above inequality shows that any Nash equilibrium achieves at least $1/2$ of the optimal social utility function value.

\begin{theorem}
\label{thm:NEC}
\cite{Vetta2002} For a valid utility system $(\gamma, \{\alpha_i\}_{i=1}^N)$, if  the social utility function $\gamma$ is nondecreasing and submodular, then for any Nash equilibrium $S\in\mathcal{S}$ we have 
\begin{equation}
\label{ineq:NEC}
\bar{\gamma}(S)\geq\frac{1}{1+c}\bar{\gamma}(\Omega).
\end{equation}
\end{theorem}

When the social utility function $\gamma$ is nondecreasing and submodular, we have $c\in[0,1]$, which implies that $
\bar{\gamma}(S)\geq \bar{\gamma}(\Omega)/2$.

\section{Nash Equilibria Based on User Groups}

\subsection{Social-Aware Nash Equilibria}

In this section, we first introduce the social group utility maximization system and the social-aware Nash equilibrium defined in \cite{Zhang2014}. Then, we show that the results of \cite{Vetta2002} are directly applicable to bounding the performance of any social-aware Nash equilibrium.

In \cite{Zhang2014}, each user belongs to a group and aims to maximize its social group utility instead of its private utility. Each group is formed based on social ties between users and may reflect friendship, kinship, college relationship, etc. The social group utility for user $i$ (a mapping from $\mathcal{X}$ to real numbers) is defined as 
\[\eta_i=\alpha_i+\sum\limits_{m\in\mathcal{N}_i^s}\omega_{im}\alpha_m\]
where $\alpha_i$'s are private utilities, $\mathcal{N}_i^s$ is the set of all users having a social tie with user $i$, and $w_{im}$'s are weight parameters that reflect the strengths of social ties between user $i$ and  the users in $\mathcal{N}_i^s$, and $w_{im}\in[0,1]$. Correspondingly, the expected group utility $\bar{\eta}_i$ for user $i$, mapping from sequences in $\mathcal{S}$ to real numbers, is the expected value of $\eta_i$. 

\begin{definition}
\label{dfn:SNE}
\cite{Zhang2014} A strategy sequence $S=(s_1,\ldots, s_N)\in\mathcal{S}$ is a \emph{social-aware Nash equilibrium} if no user can improve its group utility by unilaterally changing its strategy, i.e., for any group $i$,
\begin{equation}
\label{ineq:SNE}
\bar{\eta}_i(S)\geq \bar{\eta}_i((S_{-i}, s_i')), \quad \forall s_i'\in \mathcal{S}_i.
\end{equation}
\end{definition}

\medskip
By comparing the definition of a Nash equilibrium and a social-aware Nash equilibrium, we see that the only difference between them is  that one is defined based on expected private utility functions and the other based on expected group utility functions. But because in \cite{Zhang2014}, each user has its own group utility function, and therefore its own expected group utility function, then  the results of \cite{Vetta2002} (in particular Theorem~1 and Theorem~2) directly apply to bound the performance of the social-aware Nash equilibrium of \cite{Zhang2014}. We prove in Theorem~3 and Theorem~4 that this is in fact the case, if the social group utility system $(\gamma, \{\eta_i\}_{i=1}^N)$ is valid.
A social group utility system $(\gamma, \{\eta_i\}_{i=1}^N)$ is valid if it satisfies the following assumptions, which are counterparts of Assumption~\ref{assumption1} and Assumption~\ref{assumption2} with expected group utilities standing in for expected private utilities.

\begin{assumption}
\label{assumption3}
The group utility of user $i$ $(1\leq i\leq N)$ is at least as large as the loss in the social utility resulting from user $i$ dropping out of the game. That is, the system $(\gamma,  \{\eta_i\}_{i=1}^N)$ has the property that for any strategy sequence $S = (s_1,\ldots, s_N)\in\mathcal{S}$,
\begin{equation}
\label{ineq:ass3}
\bar{\eta}_i(S)\geq\bar{\gamma}_{s_i}(S_{-i}),\quad \forall 1\leq i\leq N.
\end{equation}
\end{assumption}

\medskip
\begin{assumption}
\label{assumption4} The sum of the group utilities of the system is not larger than the social utility, i.e., for any strategy sequence $S=(s_1,\ldots, s_N)\in\mathcal{S}$,
\begin{equation}
\label{ineq:assm4}
\sum\limits_{i=1}^N\bar{\eta}_i(S)\leq \bar{\gamma}(S).
\end{equation}
\end{assumption}

Given $X\in\mathcal{X}$, if for any $1\leq i\leq N$, the inequalities $\eta_i(X)\geq\gamma_{x_i}(X_{-i})$  and $\sum_{i=1}^N\eta_i(X)\leq \gamma(X)$ hold, then the inequalities (\ref{ineq:ass3}) and (\ref{ineq:assm4}) hold.
\begin{remark}
Comparing Definitions~\ref{dfn:NE} and \ref{dfn:SNE}, we have that the only difference between a Nash equilibrium and a social-aware Nash equilibrium is that the former is defined in terms of $\bar{\alpha}_i$, and the latter is defined in terms of $\bar{\eta}_i$. So if we take $\bar{\eta}_i$  to play the role of $\bar{\alpha}_i$, then  satisfying Assumptions~\ref{assumption3} and \ref{assumption4} means that the utility system satisfies Assumptions~\ref{assumption1} and \ref{assumption2}. Based on the results of Theorems~\ref{thm:NE} and \ref{thm:NEC}, we have the following Theorems~\ref{thm:SNE} and \ref{thm:SNEC}.
\end{remark}

\begin{theorem}
\label{thm:SNE}
For a valid utility system $(\gamma, \{\eta_i\}_{i=1}^N)$, if the social utility function $\gamma$ is submodular, then for any social-aware Nash equilibrium $S\in\mathcal{S}$ we have 
\begin{equation}
\label{ineq:NE}
\bar{\gamma}(S)\geq\frac{1}{2}\left(\bar{\gamma}(\Omega)+\sum_{i=1}^N\bar{\gamma}_{s_i}({S_{-i}\cup \Omega})\right).
\end{equation}
\end{theorem}

\begin{theorem}
\label{thm:SNEC}
For a valid utility system $(\gamma, \{\eta_i\}_{i=1}^N)$, if  the social utility function $\gamma$ is nondecreasing and submodular, then for any Nash equilibrium $S\in\mathcal{S}$ we have 
\begin{equation}
\label{ineq:NEC}
\bar{\gamma}(S)\geq\frac{1}{1+c}\bar{\gamma}(\Omega).
\end{equation}
\end{theorem}

\subsection{Group Nash Equilibria}

In this section we consider a different type of social group utility maximization system in which the set of all users are divided into disjoint groups, and the users in the same group choose their strategies by maximizing their group utility function jointly.

Assume that the set of users $\mathcal{N}=\{1,\ldots, N\}$ is divided into $l$ disjoint groups, in which group $i$ ($1\leq i\leq l$) has users $\{m_i+1,\ldots,m_i+k_i\}$, where $m_i=\sum_{j=1}^{i-1}k_j$, $k_j$ is the number of users in group $j$, and $\sum_{j=1}^lk_j=N$.  Let $s^i=(s_{m_i+1},\ldots, s_{m_i+k_i})$, where $s_i\in\mathcal{S}_i$ is the strategy for user $i$. We call $s^i$ the \emph{group strategy} for group $i$. It includes the strategies taken by all the users in group $i$ ($1\leq i\leq l$). We use $S^{-i}$ to denote the sequence of group strategies taken by all groups except for group $i$. Given $S^{-i}$, we denote by $(S^{-i}, t^i)$ the group strategy sequence obtained when group $i$ changes its group strategy from $s^i$ to $t^i$. Similarly, for $X\in\mathcal{X}$, we use $x^i$ and $X^{-i}$ to denote the sequence of actions taken by the users in group $i$, and  the sequence of actions taken by all groups except for group $i$, respectively. For convenience, we still use $\eta_i$ and $\bar{\eta}_i$ to denote the group utility function and the expected group utility function for group $i$.

We define a \emph{group Nash equilibrium} as follows.

\begin{definition}
 A strategy set $S=(s_1,\ldots,s_N)$ is a group Nash equilibrium of a utility system if no group  can improve its group utility by unilaterally changing its group strategy, i.e., for any $1\leq i\leq l$,
\[
 \bar{\eta}_i(S)\geq \bar{\eta}_i((S^{-i}, t^{i})), \quad \forall t^i=(t_{m_i+1},\ldots, t_{m_i+k_i}),
\]
where $t_j\in\mathcal{S}_j$ for $m_i+1\leq j\leq m_i+k_i$. 
 
 \end{definition}

We say that the utility system $(\gamma, \{\eta_i\}_{i=1}^l)$ is \emph{valid} if it satisfies the following two assumptions. 
 
\begin{assumption}
\label{assumption5}
The group utility of group $i$ is at least as large as the loss in the social utility resulting from all the users in group $i$ dropping out of the game. That is, the system $(\gamma, \{\eta_i\}_{i=1}^l)$  has the property that for any strategy sequence $S=(s^1,\ldots, s^l)\in\mathcal{S}$,
\begin{equation}
\label{ineq:ass5}
\bar{\eta}_i(S)\geq\bar{\gamma}_{s^i}(S^{-i}),\quad \forall 1\leq i\leq l.
\end{equation}
\end{assumption}

\medskip
\begin{assumption}
\label{assumption6}
The sum of the group utilities of the system is not larger than  the social utility, i.e., for any strategy sequence $S=(s^1,\ldots, s^l)\in\mathcal{S}$,
\begin{equation}
\label{ineq:ass6}
\sum\limits_{i=1}^l\bar{\eta}_i(S)\leq \bar{\gamma}(S).
\end{equation}
\end{assumption}

\medskip

Given $X\in\mathcal{X}$, if for any $1\leq i\leq l$, the inequalities $\eta_i(X)\geq\gamma_{x^i}(X^{-i})$  and $\sum_{i=1}^l\eta_i(X)\leq \gamma(X)$ hold, then the inequalities (\ref{ineq:ass5}) and (\ref{ineq:ass6}) hold.
We now present our results on the performance of a group Nash equilibrium relative to the optimal social strategy $\Omega$. Although the overall flow of the proof for deriving  performance bound (without curvature) for the group Nash equilibria is similar to that of the proof from \cite{Vetta2002}, we still include it here because it will help us derive performance bounds involving curvature later on.

\begin{lemma}
\label{lemma1}
Assume that the social utility function $\gamma$ is a submodular set function. Then for any strategy set $S\in\mathcal{S}$, 
\begin{equation}
\label{ineq:lemma1}
\bar{\gamma}(\Omega)\leq \bar{\gamma}(S)+\sum\limits_{i:\sigma^i\subseteq \Omega\setminus S}\bar{\gamma}_{\sigma^i}(S^{-i})-\sum\limits_{i:s^i\subseteq S\setminus \Omega}\bar{\gamma}_{s^i}({S^{(i-1)}\cup \Omega}), 
\end{equation}
where $S^{(i)}=s^1\oplus s^2\oplus\cdots\oplus s^{i}$ is the sequence of the group strategies taken by the first $i$ groups.

\end{lemma}
\vspace{1mm}
\begin{proof}Write $\Omega=\sigma^1\oplus\cdots\oplus \sigma^l$ and $S=s^1\oplus\cdots\oplus s^l$, where $\sigma^i=(\sigma_{m_i+1},\ldots, \sigma_{m_i+k_i})$, $s^i=(s_{m_i+1},\ldots, s_{m_i+k_i})$, and $\sigma_j, s_j\in\mathcal{S}_j$ for $m_i+1\leq j\leq m_i+k_i$.

By Propositions~1 and 2 in \cite{Liu2016}, we have that 
\begin{align*}
\bar{\gamma}(\Omega\cup S)&\leq \bar{\gamma}(S)+\sum\limits_{i:\sigma^i\subseteq \Omega\setminus S}\bar{\gamma}_{\sigma^i}(S)\\
&\leq \bar{\gamma}(S)+\sum\limits_{i:\sigma^i\subseteq \Omega\setminus S}\bar{\gamma}_{\sigma^i}(S^{-i})
\end{align*}
and
\[\bar{\gamma}(\Omega\cup S)=\bar{\gamma}(\Omega)
+\sum\limits_{i:s^i\subseteq S\setminus \Omega}\bar{\gamma}_{s^i}({S^{(i-1)}\cup \Omega}).\]
Combining the two inequalities above, we have (\ref{ineq:lemma1}).
\end{proof}
\begin{theorem}
\label{thm:GNE}
For a valid utility system $(\gamma, \{\eta_i\}_{i=1}^N)$, if the social utility function $\gamma$ is submodular, then  any group Nash equilibrium $S=(s_1,\ldots,s_N)\in\mathcal{S}$ satisfies
\begin{equation}
\label{ineq:GNS}
\bar{\gamma}(S)\geq \frac{1}{2}\left(\bar{\gamma}(\Omega)+\sum\limits_{i=1}^l\bar{\gamma}_{s^i}(\Omega\cup S^{-i})\right).
\end{equation}
\end{theorem}
\begin{proof}
By Lemma~\ref{lemma1}, we have 
\[\bar{\gamma}(\Omega)\leq \bar{\gamma}(S)+\sum\limits_{i:\sigma^i\subseteq \Omega\setminus S}\bar{\gamma}_{\sigma^i}(S^{-i})-\sum\limits_{i:s^i\subseteq S\setminus \Omega}\bar{\gamma}_{s^i}({S^{(i-1)}\cup \Omega}).\]
By the definition of a group Nash equilibrium, we have 
\[\sum\limits_{i:\sigma^i\subseteq \Omega\setminus S}\bar{\gamma}_{\sigma^i}(S^{-i})\leq \sum\limits_{i:\sigma^i\subseteq \Omega\setminus S}\bar{\gamma}_{s^i}(S^{-i})\leq \sum\limits_{i:s^i\subseteq S\setminus \Omega}\bar{\gamma}_{s^i}(S^{-i}).\]
By Assumptions~\ref{assumption5} and \ref{assumption6}, we have 
\begin{align*}
\sum\limits_{i:s^i\subseteq S\setminus \Omega}\bar{\gamma}_{s^i}(S^{-i})&\leq \sum\limits_{i:s^i\subseteq S\setminus \Omega}\bar{\eta}_i(S)\\
&\leq \bar{\gamma}(S)-\sum\limits_{i:s^i\subseteq S\cap \Omega}\bar{\eta}_i(S)\\
&\leq  \bar{\gamma}(S)-\sum\limits_{i:s^i\subseteq S\cap \Omega}\bar{\gamma}_{s^i}(S^{-i}).
\end{align*}
Combining the inequalities above and using submodularity results in
\begin{align*}
\bar{\gamma}(\Omega)
&\leq 2\bar{\gamma}(S)-\sum\limits_{i:s^i\subseteq S\cap \Omega}\bar{\gamma}_{s^i}(S^{-i})-\sum\limits_{i:s^i\subseteq S\setminus \Omega}\bar{\gamma}_{s^i}(\Omega\cup S^{(i-1)})\\
&\leq 2\bar{\gamma}(S)-\sum\limits_{i:s^i\subseteq S\cap \Omega}\bar{\gamma}_{s^i}(\Omega\cup S^{-i})-\sum\limits_{i:s^i\subseteq S\setminus \Omega}\bar{\gamma}_{s^i}(\Omega\cup S^{-i})\\
&\leq  2\bar{\gamma}(S)-\sum\limits_{i=1}^l\bar{\gamma}_{s^i}(\Omega\cup S^{-i}),
\end{align*}
which implies that the inequality (\ref{ineq:GNS}) holds.
\end{proof}

\begin{remark}
If the utility function $\gamma$ is nondecreasing, then the term $\sum_{i=1}^l\bar{\gamma}_{s^i}(\Omega\cup S^{-i})$ is  non-negative, so $\bar{\gamma}(S)\geq \frac{1}{2}\bar{\gamma}(\Omega)$, which means that the social value of any group Nash equilibrium is at least half of the optimal social utility value.
\end{remark}

To better characterize the relation of the social utility value of any group Nash equilibrium and that of the optimal solution $\Omega$, we define the {group curvature} $c_{k_i}$ of the social utility function for group $i$ as 
\[c_{k_i}=\max\limits_{S\in\mathcal{S}, \bar{\gamma}_{s^i}(\emptyset)\neq 0}\left\{1-\frac{\bar{\gamma}_{s^i}(\Omega\cup S^{-i})}{\bar{\gamma}_{s^i}(\emptyset)}\right\}{.}\]

\begin{lemma}
\label{curvaturecom}
Assume tha the utility function $\gamma$ is submodular and nondecreasing. Then we have $c_{k_i}\leq c$ for $1\leq i\leq l$.  Especially, if $\mathcal{X}_1=\mathcal{X}_2=\cdots=\mathcal{X}_N$, then we have $c_{k_i}\leq c_{k_j}$ for $k_i\geq k_j$.
\end{lemma}
The proof of $c_{k_i}\leq c$ is similar to that of Theorem~3.3 from \cite{Liu2017} and the proof of  $c_{k_i}\leq c_{k_j}$ for $k_i\geq k_j$ is similar to  that of  Theorem~3.4 from \cite{Liu2017}, so we skip it here. 

\begin{lemma}
\label{lemma2}
Assume that $\gamma$ is a submodular set function. Then for any strategy set $S=(s_1,\ldots,s_N)\in\mathcal{S}$, we have 
\[\bar{\gamma}(S)\leq \sum\limits_{i=1}^l\bar{\gamma}_{s^i}(\emptyset)\]
where $s^i=(s_{m_i+1},\ldots, s_{m_i+k_i})$ for $1\leq i\leq l$.
\end{lemma}
\begin{proof}
By the submodularity of $\bar{\gamma}$, we have
\begin{align*}
\bar{\gamma}(S)&=\bar{\gamma}_{s^1}(\emptyset)+\bar{\gamma}_{s^2}(s^1)+\cdots+\bar{\gamma}_{s^i}(s^1\oplus\cdots\oplus s^{i-1})\\
&\quad\quad+\cdots+\bar{\gamma}_{s^l}(s^1\oplus \cdots\oplus s^{l-1})\\
&\leq \bar{\gamma}_{s^1}(\emptyset)+\bar{\gamma}_{s^2}(\emptyset)+\cdots+\bar{\gamma}_{s^i}(\emptyset)+\cdots+\bar{\gamma}_{s^l}(\emptyset)\\
&=\sum\limits_{i=1}^l\bar{\gamma}_{s^i}(\emptyset).
\end{align*}
\end{proof}
\begin{theorem}
\label{thm:GNEC}
For a valid utility system $(\gamma, \{\eta_i\}_{i=1}^l)$, if the social utility function $\gamma$ is nondecreasing and submodular, then  any group Nash equilibrium $S=(s_1,\ldots,s_N)\in\mathcal{S}$ satisfies
\[\bar{\gamma}(S)\geq  \frac{1}{1+\max\limits_{1\leq i\leq l}c_{k_i}}\bar{\gamma}(\Omega).\]
Especially, if $\mathcal{X}_1=\mathcal{X}_2=\cdots=\mathcal{X}_N$, we have 
\[\bar{\gamma}(S)\geq  \frac{1}{1+c_{k^*}}\bar{\gamma}(\Omega),\]
where $k^*=\min_{1\leq i\leq l}k_i$.
\end{theorem}
\begin{proof}
For any group Nash equilibrium $S\in\mathcal{S}$, write $S=s^1\oplus\cdots\oplus s^l$, where $s^i=(s_{m_i+1},\ldots, s_{m_i+k_i})$ for $1\leq i\leq l$.

By the definition of the curvature $c_{k_i}$ for group $i$, we have 
\[\bar{\gamma}_{s^i}(\Omega\cup S^{-i})\geq\left(1-c_{k_i}\right)\bar{\gamma}_{s^i}(\emptyset).\]
Using the inequality above,  Lemma~\ref{lemma2}, and Theorem~\ref{thm:GNE}, we have 
\begin{align*}
\bar{\gamma}(S)&\geq \frac{1}{2}\left(\bar{\gamma}(\Omega)+\sum\limits_{i=1}^l\bar{\gamma}_{s^i}(\Omega\cup S^{-i})\right)\\
&\geq  \frac{1}{2}\left(\bar{\gamma}(\Omega)+\sum\limits_{i=1}^l\left(1-c_{k_i}\right)\bar{\gamma}_{s^i}(\emptyset)\right)\\
&\geq \frac{1}{2}\left(\bar{\gamma}(\Omega)+(1-\max_{1\leq i\leq l}c_{k_i})\sum\limits_{i=1}^l\bar{\gamma}_{s^i}(\emptyset)\right)\\
&\geq \frac{1}{2}(\bar{\gamma}(\Omega)+(1-\max_{1\leq i\leq l}c_{k_i})),
\end{align*}
which implies that 
\[\bar{\gamma}(S)\geq \frac{1}{1+\max\limits_{1\leq i\leq l}c_{k_i}}\bar{\gamma}(\Omega).\]
 When $\mathcal{X}_1=\mathcal{X}_2=\cdots=\mathcal{X}_N$, by Lemma~\ref{curvaturecom}, we have that $c_{k_i}\leq c_{k_j}$ for $k_i\geq k_j$. Therefore, we have
\[\bar{\gamma}(S)\geq \frac{1}{1+c_{k^*}}\bar{\gamma}(\Omega),\]
where  $k^*=\min_{1\leq i\leq l}k_i$.
\end{proof}
\begin{remark}
 When the group utility function $\gamma$ is non-decreasing and submodular, it is easy to check that $c_{k_i}\in[0,1]$, which implies that $1/(1+\max_{1\leq i\leq l}c_{k_i})\geq 1/2$.
\end{remark}
\begin{remark}
  When the group utility function $\gamma$ is non-decreasing and submodular, we have $\bar{\gamma}(S)\geq\bar{\gamma}(\Omega)/ (1+\max_{1\leq i\leq l}c_{k_i})\geq \bar{\gamma}(\Omega)/ (1+c)$. This shows that the bound for the case with grouping is tighter than that for the case without grouping. Of course, this is unsurprising, because grouping entails cooperation. Moreover, under the condition that each user has the same action space,  the larger the value of $k_i$, the higher the degree of cooperation, and the tighter the lower bound. 
\end{remark}
\begin{remark}
We point out that each group can be viewed as a new user with vector-valued actions, and a 1/2 bound for the performance of group Nash equilibrium follows from the result of Vetta. But our analysis goes further by defining  the group
curvature $c_{k_i}$ associated with group $i$ with $k_i$ users; in doing so, we obtain a tighter bound, namely $1/(1+\max_{1\leq i\leq l}c_{k_i})$. In the special case where each user has the same action space, then we have that any group Nash equilibrium achieves at least $1/(1+c_{k^*})$ of the optimal social utility, where $k^*$ is the least number of users among the $l$ groups, and the larger the value of $k^*$, the tighter the lower bound.

\end{remark}

\section{Example}
In this section, we consider the application of utility-based maximization in database assisted spectrum access, adopted from \cite{Zhang2014}. We will show that the utility system is valid and the social utility function is submodular. 
We then apply the performance bounds for Nash, social-aware Nash, and group Nash equilibria.

Consider a set of  users $\mathcal{N}=\{1,\ldots, N\}$ and a set of TV channels $\mathcal{M}=\{1,\ldots, M\}$. The users in $\mathcal{N}$ wish to access the TV channels in $\mathcal{M}$, for  purposes other than TV transmissions, in a way that does not unnecessarily disrupt the primary use of these channels, which is for TV transmission. Specifically, to protect the primary TV users, each user $i$ sends a spectrum access request message containing its geo-location information to a geo-location database. In response, the database sends back the set of vacant channels $\mathcal{M}_i\in\mathcal{M}$ and the allowable transmission power level $P_i$. Then each user $i$ chooses a feasible channel $a_i$ from the vacant channel set $\mathcal{M}_i$ for data transmission. When multiple users choose to access the same vacant channel, they might interfere with each other, depending on their relative distance: If the distance between users $m$ and $i$ is $d_{mi}$, interference occurs only if $d_{mi}\leq \delta$, where $\delta$ is a given threshold. The aim is to minimize the total interference which is the sum of interference received by each user. 

For a  collection of selected channels $A=(a_1,\ldots, a_N)\in\prod_{i=1}^N\mathcal{M}_i$, the interference experienced by user $i$ is defined as
\[
I_i(A)=\sum\limits_{ m\in\mathcal{N}_{i}^p}P_md_{mi}^{-\lambda}I_{\{a_i=a_m\}}+\omega_{a_i}^i,
\]
where $\mathcal{N}_i^p$ is the set of users that can interfere with user $i$, $\lambda$ is a path-loss factor, $I_{\{\cdot\}}$ is the indicator function, and $\omega_{a_i}^i$ is the noise including the interchannel interference in channel $a_i$ resulting from primary TV users using other channels.
The private utility function $\alpha_i$ of user $i$ is then defined as 
\[
\alpha_i(A)=-I_i(A)=-\sum\limits_{m\in\mathcal{N}_{i}^p}P_md_{mi}^{-\lambda}I_{\{a_i=a_m\}}-\omega_{a_i}^i.
\] 
This private utility reflects the fact that each user desires to minimize its experienced interference. The social group utility of each user $i$ is defined as 
\[
\eta_i(A)=\alpha_i(A)+\sum\limits_{m\in\mathcal{N}_i^s}w_{im}\alpha_m(A).
\]
Finally, the social utility function is $\gamma(A)=\sum_{i=1}^N\alpha_i(A)$.

\subsection{Nash Equilibria}

\label{subsectionA}

First we will prove that the utility system $(\gamma, \{\alpha_i\}_{i=1}^N)$ satisfies Assumptions~\ref{assumption1} and \ref{assumption2}, and the social utility function $\gamma(A)=\sum_{i=1}^N\alpha_i(A)$ is submodular.

To prove that  the system $(\gamma, \{\alpha_i\}_{i=1}^N)$ satisfies Assumption~\ref{assumption1}, it suffices to prove that for $1\leq i\leq N$, 
\[
\alpha_i(A)\geq \gamma(A)-\gamma(A_{-i}).
\]

By the definition of $\alpha_i(A)$, we have that
\[
\gamma(A)=-\sum\limits_{i=1}^N\sum\limits_{m\in\mathcal{N}_{i}^p}P_md_{mi}^{-\lambda}I_{\{a_i=a_m\}}-\sum\limits_{i=1}^N\omega_{a_i}^i.
\]
Thus,
\begin{align*}
\gamma(A)-\gamma(A_{-i})
&=-\sum\limits_{m\in\mathcal{N}_{i}^p}P_md_{mi}^{-\lambda}I_{\{a_i=a_m\}}-\sum\limits_{n:i\in\mathcal{N}_{n}^p}P_id_{in}^{-\lambda}I_{\{a_n=a_i\}}-\omega_{a_i}^i\\
&= \alpha_i(A)-\sum\limits_{n: i\in\mathcal{N}_{n}^p}P_id_{in}^{-\lambda}I_{\{a_n=a_i\}}\\
&\leq  \alpha_i(A),
\end{align*}
which shows that the utility system $(\gamma, \{\alpha_i\}_{i=1}^N)$ satisfies Assumption~\ref{assumption1}.
Because $\gamma(A)=\sum_{i=1}^N\alpha_i(A)$, the utility system  $(\gamma, \{\alpha_i\}_{i=1}^N)$ also satisfies Assumption~\ref{assumption2}.

Let $A_k=(a_1,\ldots, a_k)$ and $A_l=A_k\oplus (a_{k+1},\ldots, a_{l})$ $(l< N)$. To prove that $\gamma(A)=\sum_{i=1}^N\alpha_i(A)$ is submodular, it suffices to prove that for any $a_{j}\in \mathcal{M}_{j}$ ($l+1\leq j\leq N$),
\[
\gamma_{a_{j}}(A_k)\geq \gamma_{a_{j}}(A_l).
\]
By definition, we have
\begin{align*}
\gamma_{a_{j}}(A_k)
&=\gamma(A_k\oplus {a_{j}})-\gamma(A_k)\\
&=-\sum\limits_{m\in\mathcal{N}_{j}^p,1\leq m\leq k}P_md_{mj}^{-\lambda}I_{\{a_{j}=a_m\}}-\sum\limits_{n:j\in\mathcal{N}_{n}^p,1\leq n\leq k}P_{j}d_{jn}^{-\lambda}I_{\{a_n=a_{j}\}}-\omega_{a_{j}}^{j}
\end{align*}
and
\begin{align*}
\gamma_{a_{j}}(A_l)
&=\gamma(A_l\oplus {a_{j}})-\gamma(A_l)\\
&=-\sum\limits_{m\in\mathcal{N}_{j}^p, 1\leq m\leq l}P_md_{mj}^{-\lambda}I_{\{a_j=a_m\}}-\sum\limits_{n:j\in\mathcal{N}_{n}^p, 1\leq n\leq l}P_{j}d_{jn}^{-\lambda}I_{\{a_n=a_{j}\}}-\omega_{a_{j}}^{j},
\end{align*}
which implies that 
\[
\gamma_{a_{j}}(A_k)\geq \gamma_{a_{j}}(A_l).
\]

We have now established
that the utility system $(\gamma, \{\alpha_i\}_{i=1}^N)$ is valid, and the social utility function $\gamma(A)=\sum_{i=1}^N\alpha_i(A)$ is submodular. This implies that the performance bound in Theorem~\ref{thm:NE} holds.
\subsection{Social-Aware Nash Equilibria}

Let
\[
p=\min\limits_{1\leq j\leq N}\{1+\sum\limits_{i:j\in\mathcal{N}_i^s}w_{ij}\}
\]
Because maximizing $\sum_{i=1}^N\alpha_i(A)$ (with respect to $A\in\mathcal{M}$) is equivalent to maximizing $p\sum_{i=1}^N\alpha_i(A)$, for convenience, we set 
$\gamma(A)=p\sum_{i=1}^N\alpha_i(A)$ when considering the utility system 
$(\gamma, \{\eta_i\}_{i=1}^N)$. 
 
Now prove that the system satisfies  Assumption~\ref{assumption4}.
\begin{align*}
\sum\limits_{i=1}^N\eta_i(A)&=\sum\limits_{i=1}^N\alpha_i(A)+\sum\limits_{i=1}^N\sum\limits_{n:n\in\mathcal{N}_i^s}\omega_{in}\alpha_n(A)\\
&=\sum\limits_{j=1}^N(1+\sum\limits_{i:j\in\mathcal{N}_i^s}w_{ij})\alpha_j(A)\\
&\leq p\sum\limits_{i=1}^N\alpha_i(A).
\end{align*}
This implies that the utility system $(\gamma, \{\eta_i\}_{i=1}^N)$ satisfies Assumption~\ref{assumption4}.

We now prove that the utility system $(\gamma, \{\eta_i\}_{i=1}^N)$ satisfies Assumption~\ref{assumption3}.  By the definition of $\gamma(A)$ and $\eta_i(A)$, we have 
 \begin{align*}
 \begin{split}
 \gamma(A)-\gamma(A_{-i})&=p\left(-\sum\limits_{m\in\mathcal{N}_{i}^p}P_md_{mi}^{-\lambda}I_{\{a_i=a_m\}}
 -\sum\limits_{n:i\in\mathcal{N}_{n}^p}P_id_{in}^{-\lambda}I_{\{a_n=a_i\}}-\omega_{a_i}^i\right)
\end{split}\\
&=p\left(\alpha_i(A)-\sum\limits_{n:i\in\mathcal{N}_{n}^p}P_id_{in}^{-\lambda}I_{\{a_n=a_i\}}\right)\\
&=\alpha_i(A)+\min\limits_{1\leq j\leq N}\{\sum\limits_{i:j\in\mathcal{N}_i^s}w_{ij}\}\alpha_i(A)- p\sum\limits_{n:i\in\mathcal{N}_{n}^p}P_id_{in}^{-\lambda}I_{\{a_n=a_i\}}.
 \end{align*}
 and
 \begin{align*}
 \eta_i(A)&=\alpha_i(A)+\sum\limits_{n:n\in\mathcal{N}_i^s}w_{in}\alpha_n(A).
 \end{align*}
For convenience, we consider the case when the transmission power of all the users are the same (\emph{i.e.}, $P_m=P_n=P$ for any users $m$ and $n$). By Theorem~1 from \cite{Zhang2014}, we have that the social tie between any two users is symmetric (\emph{i.e.}, $w_{nm}=w_{mn}$). Then we can write $p$ and $p(\gamma(A)-\gamma(A_{-i}))$ as follows.
$$p=\min\limits_{1\leq i\leq N}\{1+\sum\limits_{m\in\mathcal{N}_i^s}w_{im}\}$$
and
\begin{align*}
p(\gamma(A)-\gamma(A_{-i}))&=p(\alpha_i(A)-\sum\limits_{m\in\mathcal{N}_i^p}Pd_{mi}^{-\lambda}I_{\{a_i=a_m\}})\\
&=\alpha_i(A)+(\min\limits_{1\leq i\leq N}\sum\limits_{m\in\mathcal{N}_i^s}w_{im})\alpha_i(A)+(-p\sum\limits_{m\in\mathcal{N}_i^p}Pd_{mi}^{-\lambda}I_{\{a_i=a_m\}}).
\end{align*}

So only if 
\begin{equation}
\label{Social-Awarecondition}
\sum\limits_{n:n\in\mathcal{N}_i^s}w_{in}\alpha_n(A)\geq (\min\limits_{1\leq i\leq N}\sum\limits_{m\in\mathcal{N}_i^s}w_{im})\alpha_i(A)-p\sum\limits_{m\in\mathcal{N}_i^p}Pd_{mi}^{-\lambda}I_{\{a_i=a_m\}}
\end{equation}
holds, we have that Assumption~\ref{assumption3} holds.

Finally, we have that $\gamma(A)=p\sum_{i=1}^N\alpha_i(A)$ is submodular because we proved that $\sum_{i=1}^N\alpha_i(A)$ is submodular in Subsection~A.
So we have now established that if the inequality (\ref{Social-Awarecondition}) holds, then the utility system  $(\gamma, \{\eta_i\}_{i=1}^N)$ is valid and the social utility function $\gamma(A)=p\sum_{i=1}^N\alpha_i(A)$ is submodular. This implies that the performance bound for a social-aware Nash equilibrium in Theorem~\ref{thm:SNE} holds.
\subsection{Group Nash Equilibria}

We now partition the set of users $\mathcal{N}=\{1,\ldots, N\}$ into $l$ 
disjoint groups and write, as before,
$\sum_{i=1}^l k_i=N$ and $m_i=\sum_{j=1}^{i-1}  k_j$. 
Group $i$ comprises the users $\{m_i+1,\ldots, m_i+k_i\}$, and the group utility function is $\eta_i(A)=\sum_{j=1}^{k_i}\alpha_{m_i+j}(A)$.
Finally, the social utility is given by $\gamma(A)=\sum_{i=1}^N\alpha_i(A)$.

We now show that the utility system $(\gamma, \{\eta_i\}_{i=1}^N)$ satisfies Assumption~\ref{assumption5}.
Let $A=a^1\oplus\cdots\oplus a^l\in\mathcal{M}$. Then for $1\leq i\leq l$,
\begin{align*}
\gamma(A)-\gamma&(A^{-i})=-\sum\limits_{j=m_i+1}^{m_i+k_i}\sum\limits_{n\in\mathcal{N}_{j}^p}P_nd_{nj}^{-\lambda}I_{\{a_j=a_n\}}\\
&\mbox{}-\sum\limits_{j=m_i+1}^{m_i+k_i}\sum\limits_{n:j\in\mathcal{N}_{n}^p}P_jd_{jn}^{-\lambda}I_{\{a_n=a_j\}}-\sum\limits_{j=m_i+1}^{m_i+k_i}\omega_{a_j}^j\\
&= \eta_i(A)-\sum\limits_{j=m_i+1}^{m_i+k_i}\sum\limits_{n:j\in\mathcal{N}_{n}^p}P_jd_{jn}^{-\lambda}I_{\{a_n=a_j\}}\\
&\leq \eta_i(A),
\end{align*}
which implies that the utility system $(\gamma, \{\eta_i\}_{i=1}^N)$  satisfies Assumption~\ref{assumption5}.

Because $\sum_{i=1}^l\eta_i(A)=\sum_{i=1}^N\alpha_i(A)=\gamma(A)$, we have that the utility system $(\gamma, \{\eta_i\}_{i=1}^N)$  also satisfies Assumption~\ref{assumption6}. Moreover, we have proved
that the social utility $\gamma(A)=\sum_{i=1}^N\alpha_i(A)$ is submodular in Subsection~A.

We have thus established that the utility system $(\gamma, \{\eta_i\}_{i=1}^N) $ is valid and the social utility function $\gamma(A)=\sum_{i=1}^N\alpha_i(A)$ is submodular. This shows that the performance bound for a group Nash equilibrium in Theorem~\ref{thm:GNE} holds.
\begin{remark}
The performance bounds we derive here for Nash equilibria, social-aware Nash equilibria, and group Nash equilibria  are worst-case performance bounds. The fact that the social-aware group Nash equilibrium derived by \cite{Zhang2014} achieves 85\% of the optimal social utility is consistent with our bound.

\end{remark}
\section{Conclusion}
 In this paper, we considered variations of the non-cooperative utility system considered by Vetta, in which users are grouped together. We considered two types of grouping among users in utility systems. The first type of grouping is from \cite{Zhang2014}, where each user belongs to a group of users having social ties with it. For this type of utility system, each user takes its strategy by maximizing its social group utility function, giving rise to the notion of social-aware Nash equilibrium.  We proved that this social utility system yields to the bounding results of Vetta  for non-cooperative system, thus establishing provable performance guarantees for the  social-aware Nash equilibria. For the second type of grouping we considered, the set of users is partitioned into $l$ disjoint groups, where the users within a group takes their group strategy by maximizing their group utility, giving rise to the notion of the group Nash equilibrium.  In this case, each group can be viewed as a new user with vector-valued actions, and a 1/2 bound for the performance of group Nash equilibria follows from the result of \cite{Vetta2002}. By defining the group
curvature $c_{k_i}$ associated with group $i$ with $k_i$ users, we showed that if the social utility function is nondecreasing and submodular, then any group Nash equilibrium achieves at least
$1/(1+\max_{1\leq i\leq l}c_{k_i})$ of the optimal social utility. Especially, if each user has the same action space, then we showed that any group Nash equilibrium achieves at least $1/(1+c_{k^*})$ of the optimal social utility, where $k^*$ is the least number of users among the $l$ groups.
Finally, we presented an example of a utility system for database assisted
spectrum access to illustrate our results.

\section*{Funding}

This work is supported in part by NSF under award CCF-1422658, and by the CSU Information Science and Technology Center (ISTeC).

\end{document}